\documentclass{amsart}
\usepackage{amsmath,amsfonts,amssymb,amsthm,epsfig,color, tikz}
\usepackage{hyperref} 
\usepackage{bbm}
\usepackage[margin=1.2in]{geometry}
\newcommand{\bb}{\mathbb}

\newcommand{\Cu}{\mathcal{C}}

\newcommand{\Z}{\bb Z}
\newcommand{\ZZ}{\bb Z}
\newcommand{\R}{\bb R}
\newcommand{\RR}{\bb R}

\newcommand{\s}{\bb S}

\newcommand{\Co}{\mathcal C}
\newcommand{\Nu}{\mathcal N}
\newcommand{\Ca}{\frak C}
\newcommand{\B}{\frak B}
\newcommand{\Sp}{\mathcal S}

\newcommand{\Pl}{\bb P}

\newcommand{\wrt}[1]{\mathrm{d}{#1}}

\newcommand{\diag}{\operatorname{diag}}
\newcommand{\dist}{\operatorname{dist}}
\newcommand{\pr}{\operatorname{pr}}
\newcommand{\spn}{\operatorname{span}}

\newcommand{\Euc}{\operatorname{Euc}}
\newcommand{\vol}{\operatorname{vol}}
\newcommand{\SL}{\operatorname{SL}}
\newcommand{\GL}{\operatorname{GL}}

\newcommand{\SO}{\operatorname{SO}}
\newcommand{\Id}{\operatorname{Id}}

\newcommand{\F}{\mathcal F}

\newcommand{\La}{\Lambda}
\newtheorem{Theorem}{Theorem}
\numberwithin{Theorem}{section}
\newtheorem{theo}[Theorem]{Theorem}

\newtheorem{coro}[Theorem]{Corollary}

\newtheorem{lemm}[Theorem]{Lemma}

\newtheorem*{lemma*}{Lemma}
\newtheorem*{question*}{Question}

\newtheorem*{theorem*}{Theorem}

\numberwithin{equation}{section}
\theoremstyle{remark}

\newtheorem{rema}[Theorem]{\sc Remark}
\begin{document}
\title[Spherical Averages]{Spherical averages of Siegel transforms for higher rank diagonal actions and applications}
\author{Jayadev~S.~Athreya}
\author{Anish Ghosh}
\author{Jimmy Tseng}

\address{J.S.A.: Department of Mathematics, University of Illinois Urbana-Champaign, 1409 W. Green Street, Urbana, IL 61801, USA}
\email{jathreya@illinois.edu}
\address{A.G.: School of Mathematics, Tata Institute of Fundamental Research, Homi Bhabha Road, Mumbai 400005 India}
\email{ghosh@math.tifr.res.in}
\address{J.T.:  School of Mathematics, University of Bristol, University Walk, Bristol, BS8 1TW UK}
\email{j.tseng@bristol.ac.uk}

    \thanks{J.S.A.\ partially supported by NSF grant DMS 1069153, and NSF grants DMS 1107452, 1107263, 1107367 ``RNMS: GEometric structures And Representation varieties" (the GEAR Network).}
    \thanks{A.G. partially supported by an ISF-UGC grant.}
    \thanks{J.T. acknowledges the research leading to these results has received funding from the European Research Council under the European Union's Seventh Framework Programme (FP/2007-2013) / ERC Grant Agreement n. 291147.}
    
    
\begin{abstract}  
We  investigate the geometry of approximates in multiplicative Diophantine approximation. Our main tool is a new averaging result for Siegel transforms on the space of unimodular lattices in $\RR^n$ which is of independent interest. 
\end{abstract}
\maketitle
\tableofcontents

\section{Introduction}
The main result in the present paper is a new averaging theorem for Siegel transforms on the homogeneous space $\SL_{n}(\R)/\SL_{n}(\Z)$. Such results have found several applications in number theory and indeed our motivation is to investigate the distribution of approximates in certain foundational results in Diophantine approximation. In \cite{AGT1}, we studied the phenomenon of \emph{spiraling} of approximates in Dirichlet's theorem and obtained a number of distribution results for approximates. In the present paper, we continue our investigations in this subject and present a new multi parameter averaging result for Siegel transforms and as a consequence, obtain new results on the geometry of approximates in \emph{multiplicative} and \emph{weighted} Diophantine approximation. We briefly recall the setup in \cite{AGT1} and then state our main results. The bulk of the paper is concerned with the proof of Theorem \ref{theorem:siegel:equidistUpper}, our result on averages of Siegel transforms. The general principle that equidistribution of spherical averages implies distribution results for approximates applies in a wide variety of situations. In the final section, we briefly survey some such situations.

\subsection{Dirichlet's theorem and spiraling}

Let $\alpha_{ij}, 1 \leq i \leq m, 1\leq j \leq n$ be real numbers and $Q > 1$. Then Dirichlet's theorem in Diophantine approximation states that there exist integers $q_1,\dots, q_m, p_1,\dots, p_n$ such that  
\begin{equation}\label{d0}
1 \leq \max\{|q_1|,\dots,|q_m|\} \leq Q
\end{equation}
\noindent and
\begin{equation}\label{d1}
\max_{1 \leq i \leq n} |\alpha_{i1}q_1 + \dots + \alpha_{im}q_m - p_i| \leq Q^{-m/n}.
\end{equation}
\noindent In our earlier work \cite{AGT1}, we studied the problem of \emph{spiraling} of the approximates appearing in Dirichlet's theorem and showed as a consequence, that on average, the directions of approximates spiral in a uniformly distributed fashion on the unit sphere of one lower dimension. In fact, the problem can be recast as a special case of a more general equidistribution result in the space of lattices. As far as we are aware, this is the only work addressing the natural question of how the approximates appearing in Dirichlet's theorem are distributed. In \cite{AGT1}, we considered vectors rather than linear forms although the proof goes through for linear forms with very minor modifications.

\noindent Given ${\bf x} \in \R^{d}$, we form the associated unimodular lattice in $\R^{d+1}$
\begin{equation}\nonumber
\Lambda_{\bf x} := \begin{pmatrix} \Id_{d} & {\bf x}\\0 & 1 \end{pmatrix} \bb Z^{d+1} = \left\{\begin{pmatrix} q{\bf x} - {\bf p}\\ q \end{pmatrix} ~:~{\bf p} \in \bb Z^{d}, q \in \bb Z\right\}.
\end{equation}

\noindent Then we can view the approximates $({\bf p}, q)$ of ${\bf x}$ appearing in Dirichlet's Theorem as points of the lattice $\Lambda_{\bf x}$ in the region
\begin{equation}\label{cone}
R := \left\{{\bf v} = \begin{pmatrix}{\bf v}_1\\ v_2 \end{pmatrix} \in \bb R^{d} \times \bb R~:~\|{\bf v}_1\||v_2|^{1/d} \leq 1\right\}.  
\end{equation} 

The set $R$ is a thinning region around the $v_2$-axis, and the following sets are used to study the distribution of lattice approximates in $R$. Let

\begin{equation}\label{defR1}
R_{\epsilon, T} := \left\{ {\bf v} \in  R~:~ \epsilon T \le v_2 \le T \right\}  
\end{equation}

\noindent and, for a subset $A$ of $\bb S^{d-1}$ with zero measure boundary,

\begin{equation}\label{defR2}
R_{A, \epsilon, T} := \left\{ {\bf v} \in R_{\epsilon, T}~:~ \frac{{\bf v}_1}{\|{\bf v}_1\|} \in A \right\}.
\end{equation} 

\noindent For a unimodular lattice $\La$, define 

$$N(\Lambda, \epsilon, T) = \#\{\Lambda \cap R_{\epsilon, T}\}$$ and 

$$N(\Lambda, A, \epsilon, T) = \#\{\Lambda \cap R_{A, \epsilon, T}\}.$$

\noindent Let $dk$ denote Haar measure on $K := K_{d+1}:=\SO_{d+1}(\bb R)$, and let $X_{d+1} := \SL_{d+1}(\bb R)/\SL_{d+1}(\bb Z)$. In \cite{AGT1}, we proved


\begin{Theorem}\label{AGT1}
For every $\Lambda \in X_{d+1}$, $A \subset \bb S^{d-1}$ as above, and for every $\epsilon > 0$, 
\begin{equation}\label{main-1}
 \lim_{T \rightarrow \infty} \frac{\int_{K} N(k^{-1}\Lambda, A, \epsilon, T)~\wrt k}{\int_{K} N(k^{-1}\Lambda, \epsilon, T)~\wrt k}= \vol(A). 
\end{equation}
\end{Theorem}

\noindent The main tool in proving Theorem \ref{AGT1} is an equidistribution result for spherical averages. Given a lattice $\Lambda$ in ${\bb R}^{d+1}$ and a bounded Riemann-integrable function $f$ with compact support on $\bb R^{d+1}$, denote by $\widehat{f}$ its \emph{Siegel transform}:

$$ \widehat{f}(\Lambda) := \sum_{\bf v \in \Lambda \backslash \{\boldsymbol{0}\}} f(\bf v).$$

\noindent Then

\begin{Theorem}\label{theorem:siegel:equidist}
Let $f$ be a bounded Riemann-integrable function of compact support on $\R^{d+1}$.  Then for any $\La \in X_{d+1}$, $$ \lim_{t \to \infty} \int_{K_{d+1}} \widehat{f}(g_t k \La) ~\wrt k = \int_{X_{d+1}}\widehat{f}~\wrt \mu.$$

\end{Theorem}

\subsection{Multiplicative and weighted variants}
Dirichlet's theorem lends itself to several interesting generalisations. Here is a \emph{multiplicative} analogue which can be proved using either Dirichlet's original approach or Minkowski's geometry of numbers. With notation as above, there exist integers $q_1,\dots, q_m, p_1,\dots, p_n$ such that  

\begin{equation}\label{md0}
 \left(\prod_{1 \leq j \leq m}\max\{1,|q_j|\}\right)^{1/m} \leq Q
\end{equation}
\noindent and
\begin{equation}\label{md1}
\left(\prod_{1 \leq i \leq n} |\alpha_{i1}q_1 + \dots + \alpha_{im}q_m - p_i|\right)^{1/n} \leq Q^{-m/n}.
\end{equation}

\noindent As a corollary, it follows that there are infinitely many $q_1,\dots, q_m$ such that
\begin{equation}\label{md2}
\left(\prod_{1 \leq i \leq n} |\alpha_{i1}q_1 + \dots + \alpha_{im}q_m - p_i|\right) \leq \left(\prod_{1 \leq j \leq m}\max\{1,|q_j|\}\right)^{-1}
\end{equation}
\noindent for some $p_1,\dots, p_n$.\\

\noindent The study of Diophantine inequalities using the multiplicative ``norm" as above instead of the supremum norm is referred to as \emph{multiplicative Diophantine approximation}. This subject is considered more difficult and is much less understood in comparison to its standard counterpart. For instance, arguably the most emblematic open problem in metric Diophantine approximation namely the Littlewood conjecture, is a problem in this genre. We refer the reader to the nice survey \cite{Bugeaud} by Bugeaud for an overview of the theory. There have been several important advances recently, several arising from applications of homogeneous dynamics to number theory. We mention the work of Kleinbock and Margulis \cite{KM} settling the Baker-Sprindzhuk conjecture as well as the work of Einsiedler-Katok-Lindenstrauss making dramatic progress towards Littlewood's conjecture.\\

Another variation of Diophantine approximation is developed as follows. Let $\alpha_{ij}, 1 \leq i \leq m, 1\leq j \leq n$ be real numbers and let ${\bf r} = (r_1, \dots, r_n) \in \R^n$ and ${\bf s} = (s_1,\dots, s_m) \in \R^m$ be probability vectors. Recall that a \textit{probability vector} has nonnegative real components, the sum of which is equal to $1$. Then a weighted version of Dirichlet's theorem states that there exist infinitely many integers $q_1,\dots, q_m$ such that
\begin{equation}\label{wd1}
\max_{1 \leq i \leq n} |\alpha_{i1}q_1 + \dots + \alpha_{im}q_m - p_i|^{1/r_i} \leq \max_{1 \leq j \leq m}|q_j|^{1/s_j}
\end{equation}
\noindent for some $p_1,\dots, p_n$. The subject of \emph{weighted} Diophantine approximation has also witnessed significant progress of late. We refer the reader to the works of Kleinbock and Weiss,  \cite{KW10, KW14} as well as the resolution of Schmidt's conjecture on weighted badly approximable vectors due to Badziahin-Pollington-Velani \cite{BDV}.\\ 

\subsection{Spiraling}
\medskip
\noindent  In this paper, we study the distribution of approximates in the multiplicative setting as well as the setting of Diophantine approximation with weights. Again, as far as we are aware, these are the first results of their kind. While our strategy remains the same as in \cite{AGT1}, our main tool, an equidistribution theorem for Siegel transforms on homogeneous spaces (Theorem \ref{theorem:siegel:equidist}) is new and new inputs are required for the proof. Equidistribution results of this kind have found many applications (cf. \cite{KM}, \cite{EMM}, \cite{MS1, MS2}) in number theory. We hope our result will be of interest to both dynamicists as well as number theorists.

\subsection{The setup}
Let $\ell\geq1$ be an integer.  Define functions $\RR^\ell \rightarrow \RR_{\geq 0}$ as follows:  \begin{align*}
\|\boldsymbol{v}\|_{\boldsymbol{p}} := \max_{i=1, \cdots, \ell} |v_i|^{1/p_i} \quad \textrm{ and  } \quad \|\boldsymbol{v}\|_{\pr} := \prod_{i=1}^\ell |v_i|
  \end{align*} where $\boldsymbol{p} \in \RR^\ell$ is a probability vector. Let $m,n \geq 1$ be integers and $d:=m+n$.  Let $e_1, \cdots, e_m$ be the standard basis for $\RR^m$ and $e_1, \cdots, e_d$ be the standard basis for $\RR^m \times \RR^n = \RR^d$.  Fix probability vectors $\boldsymbol{r} \in \RR^m$ and $\boldsymbol{s} \in \RR^n$, these vectors are also referred to as \textit{weights} in the literature.  Let 
 \[g^{(\boldsymbol{r})}_t :=\diag(e^{r_1t}, \cdots ,e^{r_m t})  \in \GL_{m}(\bb R),\] 
 \noindent and let $\bb S^{m-1}$ denote the $m-1$ dimensional unit sphere centered at the origin.  For a subset $\widetilde{A}$ of $\bb S^{m-1}$, the union of all rays in $\RR^m$ through each point of $\widetilde{A}$ is called the \textit{cone in $\RR^m$ through $\widetilde{A}$} and denoted by $\Co \widetilde{A}$.  The region of interest for Diophantine approximation with weights is \[R:= R^{(\boldsymbol{r}, \boldsymbol{s})}:=\bigg{\{}\boldsymbol{v}=\begin{pmatrix} \boldsymbol{v}_1 \\ \boldsymbol{v}_2 \end{pmatrix} \in \RR^m \times \RR^n : 0<\|\boldsymbol{v}_1\|_{\boldsymbol{r} } \|\boldsymbol{v}_2\|_{\boldsymbol{s}} \leq 1 \bigg{\}}.\]  Fix an $0<\epsilon<1,$ $T >0$, and a subset $A$ of $\bb S^{m-1}$ with zero measure boundary.  The subsets that concern us, in particular, are \[R_{\epsilon, T} := \left\{ \boldsymbol{v} \in  R~:~ \epsilon T \le \|\boldsymbol{v}_2\|_{\boldsymbol{s}} \le T \right\} \quad \textrm{ and } \quad R_{A, \epsilon, T} := \left\{ \boldsymbol{v}  \in R_{\epsilon, T}~:~ \boldsymbol{v}_1\in g^{(\boldsymbol{r})}_{-\log(T)} (\Co A) \right\}.\]  The subset  $R_{\epsilon, T}$ is analogous to the subset above which played a role in~\cite{AGT1}.  Indeed if we consider the special case of $\boldsymbol{r}$ equal to $(1/m, \cdots, 1/m)$, then the set $R_{A, \epsilon, T}$ is equal to $\left\{ \boldsymbol{v}  \in R_{\epsilon, T}~:~ \frac{\boldsymbol{v}_1}{\|\boldsymbol{v}_1\|_{2} } \in A \right\}$, which was considered in~\cite{AGT1}. The reason that our formulation in terms of cones is the appropriate generalization is as follows.  Let us again consider an arbitrary $\boldsymbol{r}$.  Consider the slices of $R$ given by the equations \[\|\boldsymbol{v}_1\|_{\boldsymbol{r} } = 1/p\] for a real number $p>1$.  To map the slice given by $p$ to the one given by $p'\geq p$, apply the contracting (and, in general, nonuniformly contracting) automorphism $g_{\log(p)-\log(p')}^{(\boldsymbol{r})}$ to the slice.  Now $g_{-t}^{(\boldsymbol{r})}$ takes $\bb S^{m-1}$ into ellipsoids, whose eccentricities are increasing as $t$ increases.  It is reasonable that the distribution of directions respects the action of $g_{-t}^{(\boldsymbol{r})}$---that this holds is the content of our result, Theorem~\ref{thmWeightedDASphereAve}.\\  

\noindent The regions of interest for multiplicative Diophantine approximation are \[P:=\bigg{\{}\boldsymbol{v}=\begin{pmatrix} \boldsymbol{v}_1 \\ \boldsymbol{v}_2 \end{pmatrix} \in \RR^m \times \RR^n : 0<\|\boldsymbol{v}_1\|_{\pr } \|\boldsymbol{v}_2\|_{\pr} \leq 1 \bigg{\}},\] \[P_{\epsilon, T} := \left\{ \boldsymbol{v} \in  P~:~ \epsilon T \le \|\boldsymbol{v}_2\|_{\pr} \le T \right\} \quad \textrm{ and } \quad P_{A, \epsilon, T} := \left\{ \boldsymbol{v}  \in P_{\epsilon, T}~:~ \boldsymbol{v}_1 \in g^{(\boldsymbol{r})}_{-\log(T)} (\Co A)  \right\}.\]  The region $P$ is sometimes referred to as a \textit{star body}.  For the special case of $\boldsymbol{r}$ equal to $(1/m, \cdots, 1/m)$, the set $P_{A, \epsilon, T}$ is equal to $\left\{ \boldsymbol{v}  \in P_{\epsilon, T}~:~ \frac{\boldsymbol{v}_1}{\|\boldsymbol{v}_1\|_{2} } \in A \right\}.$  Now, unlike for Diophantine approximation with weights, the $m$-volume of $P_{1,1}$ is infinite.  Let $\Pl_i$ denote the coordinate codimension-one hyperplane in $\RR^m$ normal to $e_i$.  Then \[\Pl_i \cap \s^{m-1} =: \s_i\] are \textit{great spheres of $\s^{m-1}$}; namely, $\s_i = k_i \s^{m-2}$ for some $k_i \in \SO_m(\RR)$.  For any $\delta>0$, let \[\s^{(\delta)}_i := \Pl_i \times [-\delta, \delta] \cap \s^{m-1}\] denote the $\delta$-thickening of $\s_i$ on $\s^{m-1}$.  By elementary calculus, it is easy to see that the $\Pl_i$ point in the directions in which $P_{1,1}$ has regions with infinite volume (see also the Appendix).  Radially projecting $P_{1,1}$ onto \[\Sp:=\Sp(\delta) := {\bb S}^{m-1} \backslash \cup_{i=1}^m \s^{(\delta)}_i\] it is easy to see that $\Co \Sp \cap P_{1,1}$ has finite $m$-volume for every $\delta>0$.  We also note that the  $g_{-t}^{(\boldsymbol{r})}$-action contracts slices of $P$ in the same way as it does $R$ and that it preserves each of the coordinate planes: $g_{-t}^{(\boldsymbol{r})}(\Pl_i) = \Pl_i$, and consequently, the action of $g_{-t}^{(\boldsymbol{r})}$ on $\Co \Sp \cap P_{1,1}$ keeps the $m$-volume finite.  By continuity in $t$, the $m$-volume of slices between $\epsilon T\leq t\leq T$ has a maximum for all fixed $1\geq \epsilon >0$ and $T>0$, and Riemann-integration implies that \begin{align}\label{eqnFiniteVolSMAD}
 \vol_{\RR^d}(P_{\Sp, \epsilon, T})< \infty.
\end{align}  For Theorem~\ref{thmMultiDASphereAve} below, we will only consider the sets \[P_{\Sp, \epsilon, T} \quad \textrm{ and } \quad P_{A, \epsilon, T}\] for $A$ with zero measure boundary contained in $\Sp(\delta)$ for some $\delta >0$.  For Theorem~\ref{thmMultiDASphereAveCuspSet}, we will consider some sets outside of $\Sp$.


\subsection{Statement of results for lattice approximates}

Let $\wrt k$ denote the probability Haar measure on $K := K_d := \SO_d(\RR)$.  Our main number-theoretic results are three averaged spiraling of lattice approximates results, one for approximation in the setting of Diophantine approximation with weights and two in the setting of multiplicative Diophantine approximation.  We point out that our proof of Theorems~\ref{thmWeightedDASphereAve} and~\ref{thmMultiDASphereAve} gives that the equality of the numerator and the equality of the denominator hold independently.  One consequence is that other ratios may be obtained.  

\begin{theo}\label{thmWeightedDASphereAve} For every unimodular lattice $\Lambda \in X_d$, subset $A \subset \s^{m-1}$ with zero measure boundary, and $\epsilon >0$, we have that \[ \lim_{T \to \infty} \frac{\int_{K}\#\{k \Lambda \cap R_{A, \epsilon, T}\}~\wrt k}{\int_{K}  \#\{k \Lambda \cap R_{\epsilon, T}\}~\wrt k}  = \frac{\vol_{\RR^d}(R_{A, \epsilon,1})}{\vol_{\RR^d}(R_{\epsilon,1})} \]
 
\end{theo}

\begin{rema}
The special case of setting $\boldsymbol{r}$ equal to $(1/m, \cdots, 1/m)$ is, itself, already a generalization of~\cite[Theorem~1.4]{AGT1}, except, since the function $\|\cdot\|_{(1/m, \cdots, 1/m)}$ is (a power of) the sup norm, instead of the Euclidean norm of ~\cite[Theorem~1.4]{AGT1}.  Here, we obtain that the limit of the ratio is \[\frac{\vol_{\RR^m}(R_{1,1} \cap \Co A)}{\vol_{\RR^m}(R_{1,1})},\]  where $\vol_{\RR^m}(R_{1,1}) = 2^m$.  Note, as mentioned, the sets $R_{A, \epsilon, T}$ for the special case reduce to their counterparts in~\cite{AGT1}.  

To obtain the exact generalization of~\cite[Theorem~1.4]{AGT1}, replace the function $\|\cdot\|_{(1/m, \cdots, 1/m)}$ by the Euclidean norm.  Then the proof of the theorem will also give this generalization and the conclusion is that the limit of the ratios is $\vol_{{\bb S}^{m-1}}(A)$.  Note that, in all cases, the function $\|\cdot\|_{\boldsymbol{s}}$ can be for an arbitrary probability $n$-vector $\boldsymbol{s}$.  We now state our results in the setting of multiplicative Diophantine approximation.

\end{rema}

\begin{theo}\label{thmMultiDASphereAve} For every unimodular lattice $\Lambda \in X_d$, $\delta >0$, subset $A \subset \Sp(\delta)=:\Sp$ with zero measure boundary, and $\epsilon >0$, we have that \[ \lim_{T \to \infty} \frac{\int_{K}\#\{k \Lambda \cap P_{A, \epsilon, T}\}~\wrt k}{\int_{K}  \#\{k \Lambda \cap P_{\Sp,\epsilon, T}\}~\wrt k}  = \frac{\vol_{\RR^d}(P_{A, \epsilon,1})}{\vol_{\RR^d}(P_{\Sp,\epsilon,1})} \]

\end{theo}

\begin{theo}\label{thmMultiDASphereAveCuspSet}  For every unimodular lattice $\Lambda \in X_d$ and open subset $A \subset \s^{m-1}$ such that \[A \cap (\cup_{i=1}^m \s_i) \neq \emptyset,\] we have that \[\lim_{T \to \infty} \int_{K}\#\{k \Lambda \cap P_{A, \epsilon, T}\}~\wrt k = \infty.\]
 
\end{theo}

\noindent Theorem~\ref{thmMultiDASphereAveCuspSet} tells us that on average there are arbitrarily small neighborhoods of directions (which we know explicitly) for which every unimodular lattice has infinitely many elements in our star body. To prove these theorems, we need our main ergodic result on equidistribution of Siegel translates, Theorem~\ref{theorem:siegel:equidist}. We note that the spiraling results for multiplicative and weighted Diophantine approximation follow by applying the Theorems above to the unimodular lattice 
$$\begin{pmatrix}\Id_{m \times m} & \alpha\\0 & \Id_{n \times n} \end{pmatrix}\bb Z^{d}$$
\noindent attached to a matrix $\alpha = (\alpha_{ij})$ as usual.

\subsubsection*{Acknowledgements} Part of this work was completed during the Group Actions and Number theory (GAN)  programme at the Isaac Newton Institute for Mathematical Sciences. We thank the INI for providing a nice venue.

\section{Equidistribution on the space of lattices}\label{sec:lattices}


Given a unimodular lattice $\Lambda$ in ${\bb R}^{d}$ and a bounded Riemann-integrable function $f$ with compact support on $\bb R^{d}$, denote by $\widehat{f}$ its \emph{Siegel transform}\footnote{One could define the Siegel transform only over primitive lattice points, in which case results analogous to Theorems~\ref{theorem:siegel:equidist} and~\ref{theorem:siegel:equidistUpper} also hold (using, essentially, the same proof).}:

$$ \widehat{f}(\Lambda) := \sum_{\bf v \in \Lambda \backslash \{\boldsymbol{0}\}} f(\bf v).$$

\noindent Let $\mu = \mu_{d}$ be the probability measure on $X_{d} :=  \SL_{d}(\bb R)/\SL_{d}(\bb Z)$ induced by the Haar measure on $\SL_{d}(\R)$ and $\wrt {\bf v}$ denote the usual volume measure on ${\bb R}^{d}$. We recall the classical Siegel Mean Value Theorem \cite{Siegel}:

\begin{Theorem}
Let $f$ be as above.\footnote{This condition can be generalized to $f \in L^1(\bb R^{d})$.} Then $\widehat{f} \in L^{1}(X_{d}, \mu)$ and
$$ \int_{{\bb R}^{d}} f ~\wrt {\bf v} = \int_{X_{d}} \widehat{f} ~\wrt \mu.$$
\end{Theorem}


\medskip 

\noindent  Note that if $f$ is the indicator function of a set $A \backslash \{\boldsymbol{0}\}$, then $\hat{f}(\Lambda)$ is simply the number of points in $\Lambda \cap (A \backslash \{\boldsymbol{0}\})$.  Let \[g_t:=g^{(\boldsymbol{r}, \boldsymbol{s})}_t :=\diag(e^{r_1t}, \cdots ,e^{r_m t},e^{-s_1t}, \cdots,e^{-s_n t})  \in \SL_{d}(\bb R)\] and $e_1, \cdots, e_{d}$ be the standard basis of $\bb R^{d}$.  We use $\widehat{{\bf 1}}_A$ to denote the indicator function of the set $A$.

\noindent Setting $t$ so that $e^{t} = T$ gives
$$g_t R_{\epsilon, T} = R_{\epsilon, 1} =: R_{\epsilon} \quad \textrm{ and } \quad g_t P_{\epsilon, T} = P_{\epsilon, 1} =: P_{\epsilon}$$
\noindent and
$$ g_t R_{A, \epsilon, T} = R_{A, \epsilon, 1} =: R_{A, \epsilon} \quad \textrm{ and } \quad g_t P_{A, \epsilon, T} = P_{A, \epsilon, 1} =: P_{A, \epsilon}.$$


\noindent Given a unimodular lattice $\Lambda \in \SL_d(\RR) / \SL_d(\ZZ)$,  a simple computation shows that 

\begin{equation}\label{eq:Sphere} \#\{k \Lambda \cap R_{\epsilon, T}\} = \widehat{{\bf 1}}_{R_{\epsilon}}(g_{t} k \Lambda) \quad \textrm{ and } \quad  \#\{k \Lambda \cap P_{\Sp, \epsilon, T}\} = \widehat{{\bf 1}}_{P_{\Sp, \epsilon}}(g_{t} k \Lambda) \end{equation}

\begin{equation}\label{eq:A} \#\{k \Lambda \cap R_{A, \epsilon, T}\} = \widehat{{\bf 1}}_{R_{A,\epsilon}}(g_{t} k \Lambda) \quad \textrm{ and } \quad  \#\{k \Lambda \cap P_{A,\epsilon, T}\} = \widehat{{\bf 1}}_{P_{A,\epsilon}}(g_{t} k \Lambda). \end{equation}

\subsection{Statement of results for Siegel transforms}\label{sec:siegel:equidist} To prove Theorems~\ref{thmWeightedDASphereAve} and~\ref{thmMultiDASphereAve}, we need to show the equidistribution of the Siegel transforms of the sets $R_{A, \epsilon}$, $P_{A, \epsilon}$, $R_{\epsilon}$, and $P_{\Sp,\epsilon}$ with respect to averages over $g^{(\boldsymbol{r}, \boldsymbol{s})}_t$-translates of $K$.  The main ergodic tool in this setting is our fourth main theorem, a result on the mutiparameter spherical averages of Siegel transforms:

\begin{Theorem}\label{theorem:siegel:equidist}
Let $f$ be a bounded Riemann-integrable function of compact support on $\R^{d}$.  Then for any $\La \in X_{d}$, $$ \lim_{t \to \infty} \int_{K_{d}} \widehat{f}(g^{(\boldsymbol{r}, \boldsymbol{s})}_t k \La) ~\wrt k = \int_{X_{d}}\widehat{f}~\wrt \mu.$$

\end{Theorem} 

The above theorem is the generalization to the multiparameter case of our theorem for the single parameter case~\cite[Theorem~2.2]{AGT1}.  Unlike in the single parameter case where the proof can be assembled from the work of Kleinbock-Margulis~\cite[Appendix]{KM}, the multiparameter case cannot, as far as we are aware.  Instead, we generalize our proof of~\cite[Theorem~2.2]{AGT1}.  As in~\cite{AGT1}, the substantial part of the argument lies in the upper bound.

\begin{Theorem}\label{theorem:siegel:equidistUpper}
Let $f$ be a bounded function of compact support in  $\R^{d}$ whose set of discontinuities has zero Lebesgue measure.  Then for any $\La \in X_{d}$, $$ \lim_{t \to \infty} \int_{K_{d}} \widehat{f}(g^{(\boldsymbol{r}, \boldsymbol{s})}_t k \La) ~\wrt k \leq \int_{X_{d}}\widehat{f}~\wrt \mu.$$ 
\end{Theorem} 

\begin{rema}
The assumption that $f$ has compact support can be replaced with that of $f \in  L^1(\bb R^{d})$---the other assumptions are still, however, necessary for the proof.
\end{rema}

\begin{coro}
Let $f$ be a bounded Riemann-integrable function of compact support in  $\R^{d}$.  Then for any $\La \in X_{d}$, $$ \lim_{t \to \infty} \int_{K_{d}} \widehat{f}(g^{(\boldsymbol{r}, \boldsymbol{s})}_t k \La) ~\wrt k \leq \int_{X_{d}}\widehat{f}~\wrt \mu.$$ 
\end{coro}

\begin{proof}
Immediate from the theorem and the Lebesgue criterion.
\end{proof}

As mentioned in~\cite{AGT1}, the lower bound follows either from the methods in \cite{KleinMarg} or by applying the following  equidistribution theorem (Theorem~\ref{theorem:DRS}) of Duke, Rudnick and Sarnak (cf. \cite{DRS}) (see also Eskin and McMullen \cite{EMc} and Shah \cite{Shah}) and then approximating the Siegel transform $\widehat{f}$ from below by $h \in C_c(X_{d})$. 
\begin{Theorem}\label{theorem:DRS}
Let $G$ be a non-compact semisimple Lie group and let $K$ be a maximal compact subgroup of $G$. Let $\Gamma$ be a lattice in $G$, let $\lambda$ be the probabilty Haar measure on $G/\Gamma$, and let $\nu$ be any probability measure on $K$ which is absolutely continuous with respect to a Haar measure on $K$. Let $\{a_n\}$ be a sequence of elements of $G$ without accumulation points. Then for any $x \in G/\Gamma$ and any $h \in C_{c}(G/\Gamma)$,
$$ \lim_{n \to \infty} \int_{K} h(a_n k x)~\wrt\nu(k) = \int_{G/\Gamma}h~\wrt \lambda.$$ 
\end{Theorem} 

\begin{rema}
One can replace $\wrt k$ by $\wrt\nu(k)$ in Theorems~\ref{theorem:siegel:equidist}~and~\ref{theorem:siegel:equidistUpper} without any changes to the proofs.

\end{rema}


\subsection{Proof of Theorems~\ref{thmWeightedDASphereAve} and~\ref{thmMultiDASphereAve}}\label{secProofMainThms}  We prove Theorems~\ref{thmWeightedDASphereAve} and~\ref{thmMultiDASphereAve} using Theorem~\ref{theorem:siegel:equidist}, while deferring the proof of the latter to Section~\ref{secSiegelEquidisProof}.  Thus, applying Theorem~\ref{theorem:siegel:equidist} to the indicator function of $R_{A, \epsilon}$, we obtain

$$ \lim_{t \to \infty} \int_{K} \widehat{{\bf 1}}_{R_{A, \epsilon}}(g^{(\boldsymbol{r}, \boldsymbol{s})}_t k \Lambda)\wrt k =\int_{X_{d}} \widehat{{\bf 1}}_{R_{A, \epsilon}}\wrt \mu = \vol_{\RR^d}(R_{A, \epsilon}),  $$

\noindent where we have applied Siegel's mean value theorem in the last equality.\footnote{A proof that $\widehat{{\bf 1}}_{R_{A, \epsilon}}, \widehat{{\bf 1}}_{R_{\epsilon}}, \widehat{{\bf 1}}_{P_{A, \epsilon}}, \widehat{{\bf 1}}_{P_{\Sp, \epsilon}}$ are Riemann-integrable is analogous to that in~\cite[Footnote~4]{AGT1}.}  Doing likewise for $R_{\epsilon}$, $P_{A, \epsilon}$, and $P_{\Sp,\epsilon}$, we obtain

\begin{align*}
 \lim_{T \to \infty} \frac{\int_{K}\#\{k \Lambda \cap R_{A, \epsilon, T}\}~\wrt k}{\int_{K}  \#\{k \Lambda \cap R_{\epsilon, T}\}~\wrt k}  &= \frac{\vol_{\RR^d}(R_{A, \epsilon})}{\vol_{\RR^d}(R_{\epsilon})} \\
 \lim_{T \to \infty} \frac{\int_{K} \#\{k \Lambda \cap P_{A,\epsilon, T}\}~\wrt k}{\int_{K} \#\{k \Lambda \cap P_{\Sp, \epsilon, T}\}~\wrt k}  &= \frac{\vol_{\RR^d}(P_{A, \epsilon})}{\vol_{\RR^d}(P_{\Sp, \epsilon})},
 \end{align*}\noindent which proves our desired results.  Note that (\ref{eqnFiniteVolSMAD}) with $T=1$ gives that $\vol_{\RR^d}(P_{\Sp, \epsilon})<\infty$.

\subsection{Proof of Theorem~\ref{thmMultiDASphereAveCuspSet}}  As in the Section~\ref{secProofMainThms}, we use Theorem~\ref{theorem:siegel:equidist} before its proof.  Let $\{\delta_i\}$ be a sequence of positive real numbers decreasing to $0$.  Then \[A \supset \cup_i A \cap \Sp(\delta_i).\]  Let $\Co_i := \Co (A \cap \Sp(\delta_i))$.  Applying Theorem~\ref{theorem:siegel:equidist}, we have \[\lim_{T \to \infty} \int_{K}\#\{k \Lambda \cap P_{A \cap \Sp(\delta_i), \epsilon, T}\}~\wrt k  = \vol_{\RR^d}(P_{A \cap \Sp(\delta_i), \epsilon}) =O\bigg(\vol_{\RR^m}(\Co_i)\bigg),\] which $\rightarrow \infty$ as $i \rightarrow \infty$ by Lemma~\ref{lemmInfiniteVolCusp}.

\section{Proof of Theorem~\ref{theorem:siegel:equidistUpper}}\label{secSiegelEquidisProof}  We adapt our proof in~\cite[Section~3]{AGT1} from the single parameter (i.e. the diagional action is $\RR$-rank $1$) case to the multiparameter (i.e. the diagional action is any allowed $\RR$-rank) case.  Recall our diagional action is \[g^{(\boldsymbol{r}, \boldsymbol{s})}_t =:g_t.\] As mentioned, to prove Theorem~\ref{theorem:siegel:equidist}, we need only show the upper bound (Theorem~\ref{theorem:siegel:equidistUpper}):  \begin{equation}
 \lim_{t \to \infty} \int_{K_{d}} \widehat{f}(g_t k \La) ~\wrt k \leq \int_{X_{d}}\widehat{f}~\wrt \mu.  
\end{equation}


\noindent Fix a unimodular lattice $\Lambda \in X_d$. The strategy of the proof is to approximate using step functions on balls. We will divide the proof into four types of multiparameter actions:  \begin{enumerate}
\item $\boldsymbol{r} := (1/m, \cdots, 1/m)$ and $n=1$.
\item $\boldsymbol{r}$  is an arbitrary probability $m$-vector and $n=1$.
\item $\boldsymbol{r}$ is an arbitrary probability $m$-vector  and $\boldsymbol{s}$ is an probability $n$-vector such there exist a unique entry $j$ for which $s_j = \|\boldsymbol{s}\|$ where $\|\cdot \|$ is the sup norm.
\item $\boldsymbol{r}$ is an arbitrary probability $m$-vector and $\boldsymbol{s}$ is an arbitrary probability $n$-vector.
\end{enumerate}

\noindent The first type is just our single parameter case~\cite[Theorem~2.2]{AGT1}.


\subsection{Proof for the second type of multiparameter}\label{sec2typeMultpara}  In this section, $\boldsymbol{r}$  is an arbitrary probability $m$-vector and $n=1$.

Using~\cite[Section~3.4]{AGT1} without change, we will approximate using step functions on balls, where we use the norm on $\R^{d} = \R^m \times \R$ given by the maximum of the Euclidean norm in $\R^m = \mbox{span}(e_1, \cdots, e_m)$ and the absolute value in $\R = \mbox{span}(e_{d})$.  Hence, balls will be open regions of $\RR^{d}$, which we also refer to as \textit{rods} or \textit{solid cylinders}.  As in~\cite{AGT1}, we need four cases:  balls centered at $\boldsymbol{0} \in \RR^{d}$, balls centered in $\spn(e_{d}) \backslash\{\boldsymbol{0}\}$, balls centered in $\spn(e_1, \cdots, e_m) \backslash\{\boldsymbol{0}\}$, and all other balls.  Since we will approximate using step functions, it suffices (as we had shown in~\cite[Section~3.4]{AGT1}) to assume that the balls in the second case do not meet $\boldsymbol{0}$ and in the last case do not meet $\spn(e_{d}) \cup \spn(e_1, \cdots, e_m)$.\footnote{We note that the second and the fourth cases already suffice to show Theorems~\ref{thmWeightedDASphereAve},~\ref{thmMultiDASphereAve}, and~\ref{thmMultiDASphereAveCuspSet}.\label{FootnoteCase24}}    Let $E := B(\boldsymbol{w}, r)$ be any such ball and $\chi_E$ be its characteristic function.  By the monotone convergence theorem, we have \[\int_{K_{d}} \widehat{\chi}_E(g_t k \La) ~\wrt k =\sum_{\boldsymbol{v} \in \La \backslash \{\boldsymbol{0}\} }\int_{K_{d}} \chi_{k^{-1} g_t^{-1} E}(\boldsymbol{v}) ~\wrt k.\]  

It is more convenient to prove the second and fourth cases together and before the others.  Let $E$ be a rod in either of these two cases.  Let $r$ be small.  Let $R:=e^t$.  Fix $R$, or equivalently, $t$ be a large value.  Now $g_t^{-1} E$ is also a rod, but narrow in the directions given by $\RR^m$ and long in the direction given by $e_d$.  Recall from~\cite[Section~3]{AGT1}, we have  \begin{align}\label{eqnInvarofRotationonSphere}\int_{K_{d}} \chi_{k^{-1} g_t^{-1} E}(\boldsymbol{v}) ~\wrt k  =: A_R^E(\|\boldsymbol{v}\|)\end{align} and \begin{align}
A_R^E(\tau) = \frac {\vol_{\tau \s^m}(\tau \s^m \cap g_t^{-1} E)}{\vol_{\tau\s^m}(\tau\s^m)}.  
\end{align}  Also, recall, from~\cite[Section~3]{AGT1}, the definition of a \textit{cap} $\Ca(\tau)$ , namely it is the intersection of the rod $g_t^{-1} E$ with the sphere $\tau \s^{m}$.  Now, unlike in~\cite{AGT1}, the caps are no longer spherical, but, for fixed $R$, are ellipsoidial of fixed eccentricity.  All our geometric considerations are for a fixed $R$ (which is only allowed to $\rightarrow \infty$ at the end).  In particular, $A_R^E(\tau)$ is a strictly decreasing smooth function with respect to $\tau$.  Let $B_{\Euc}(\boldsymbol{0}, \tau)$ denote a ball of radius $\tau$ in $\RR^{d}$ with respect to the Euclidean norm.  Now it follows from the formula for $A_R^E$ that \[\sum_{\boldsymbol{v} \in \La \backslash \{\boldsymbol{0}\} } A_R^E(\|\boldsymbol{v}\|) \leq \int_{\tau_-}^{\tau_+} \#\big(B_{\Euc}(\boldsymbol{0}, \tau) \cap\La \backslash \{\boldsymbol{0}\}\big)   ~(-\wrt A_R^E(\tau))\] where the integral is the Riemann-Stieltjes integral and the integrability of the function $\#\big(B_{\Euc}(\boldsymbol{0}, \tau) \cap\La \backslash \{\boldsymbol{0}\}\big)$ follows from its monotonicity and the continuity and monotoncity of $A_R^E(\tau)$.  The rest of the proof is identical to that in~\cite[Section~3]{AGT1} and shows \[ \lim_{t \rightarrow \infty} \int_{K_{d}} \widehat{\chi}_E(g_t k \La) ~\wrt k \leq\vol_{\RR^d}(E).\]

Finally, we prove the first and third case together.  Let $E$ be a rod in either of these two cases.  The difference between these two cases and the second and fourth cases is that the rod extends in both the positive $e_d$ and negative $e_d$ directions.  As the lattice $\Lambda$ is fixed, there is a ball $B_{\Euc}(\boldsymbol{0}, \tau_0))$ in $\RR^d$ that does not meet $\Lambda \backslash \{\boldsymbol{0}\}$ for some $\tau_0>0$ depending only on $\Lambda$.  Therefore, we can consider the two ends separately.  The proof is the same as in~\cite[Section~3.3]{AGT1}, except that $\mathfrak{B}$ is not a sphere, but an elliposoid of fixed eccentricity depending on $R$ (which, recall is fixed until the end of the proof), but this does not affect the proof.  Consequently, for the second type of multiparameter, we can conclude \[ \lim_{t \rightarrow \infty}  \int_{K_{d}} \widehat{\chi}_E(g_t k \La) ~\wrt k \leq\vol_{\RR^d}(E).\]

\subsection{Proof for the third type of multiparameter}\label{sec3typeMultpara}   In this section, $\boldsymbol{r}$ is an arbitrary probability $m$-vector  and $\boldsymbol{s}$ is an probability $n$-vector such there exist a unique entry $j$ for which $s_j = \|\boldsymbol{s}\|$ where $\|\cdot \|$ is the sup norm.   On the other hand, for the rods that we define for this multiparameter type, we will use the norm on $\R^{d} = \R^m \times \R^n$ given by the maximum of the Euclidean norm in $\mbox{span}(e_1, \cdots, e_{m+j-1}, e_{m+j+1}, \cdots e_{d})$ and the absolute value in $\mbox{span}(e_{m+j})$.  As before, we have four cases:  balls centered at $\boldsymbol{0} \in \RR^{d}$, balls centered in $\spn(e_{m+j}) \backslash\{\boldsymbol{0}\}$, balls centered in $\spn(e_1, \cdots, e_{m+j-1}, e_{m+j+1}, \cdots e_{d})\backslash\{\boldsymbol{0}\}$, and all other balls (again, Footnote~\ref{FootnoteCase24} applies).  Again, we may assume that the balls in the second case do not meet $\boldsymbol{0}$ and in the last case do not meet $\spn(e_{m+j}) \cup \spn(e_1, \cdots, e_{m+j-1}, e_{m+j+1}, \cdots e_{d})$. 

Now $g_t^{-1}$ has a unique largest expanding direction, namely along $e_{m+j}$.  Replace the role of $e_d$ from Section~\ref{sec2typeMultpara} with $e_{m+j}$.  Let $R = e^{s_jt}$.  Fix a large $R$, then the analysis of the geometry of $g_t^{-1} E$ is analogous to that in Section~\ref{sec2typeMultpara} because, for a fixed large $R$, the rod is much longer in along the $e_{m+j}$ direction than any other.  The only difference is that there exists a minimum sphere radius $\widetilde{\tau}(R)$ larger than which the analysis of the geometry is valid because some directions are expanding (but less than in $e_{m+j}$).  However, for $R$ large, $\widetilde{\tau}(R)$ is small in comparision to the length of the rod $\tau_+(R)$ (which is on the order of $R$).  In particular, $\lim_{R \rightarrow \infty} \widetilde{\tau}(R)/ \tau_+(R) =0$.  Consequently (as shown in~\cite[Section~3.3]{AGT1}  for example, the error is $O(R^{-1})$, which does not affect the proof.  The conclusion, in all four cases, is \[\lim_{t \rightarrow \infty} \int_{K_{d}} \widehat{\chi}_E(g_t k \La) ~\wrt k \leq\vol_{\RR^d}(E).\]

\subsection{Proof for the fourth type of multiparameter}  In this section, $\boldsymbol{r}$ is an arbitrary probability $m$-vector  and $\boldsymbol{s}$ is an arbitrary probability $n$-vector.  We may assume without loss of generality that there exist indices $1 \leq j_1< \cdots< j_\ell \leq n$ such that $s_{j_1} = \cdots = s_{j_\ell} = \|\boldsymbol{s}\|=:\lambda$ and $2\leq \ell \leq n$.  (Again $\|\cdot\|$ denotes the sup norm.)  Let $\widetilde{\lambda}$ denote the largest component of $\boldsymbol{s}$ strictly less than $\lambda$, or, if no such component exists, set $\widetilde{\lambda}=1$.  Let us denote this set of indices $J$ and the remaining indices by $J^c$, and note that $J \sqcup J^c = \{1, \cdots, n\}$.  The main difference and problem with this case is that caps are no longer relatively small in relation to the largest dimension of the rod.  To take care of this problem, we adapt the proof in Section~\ref{sec3typeMultpara} in two ways, the first for the analog of first and third cases and the second for the analog of the second and fourth cases.  

We use two types of balls/rods.  For the balls/rods that we define for this multiparameter type for the first and third case, we will use the norm on $\R^{d} = \R^m \times \R^n$ given by the maximum of the Euclidean norm in \[\mbox{span}(\bigcup_{i=1}^m e_i \cup \bigcup_{j \in J^c} e_{m+j})\] and the sup norm in \[\mbox{span}(\bigcup_{j \in J} e_{m+j}).\]  For the balls/rods that we define for this multiparameter for the second and fourth cases, we use the sup norm until almost the end of the proof (again, Footnote~\ref{FootnoteCase24} applies).  As before, we have the four cases:  balls centered at $\boldsymbol{0} \in \RR^{d}$, balls centered in $\mbox{span}(\bigcup_{j \in J} e_{m+j}) \backslash\{\boldsymbol{0}\}$, balls centered in $\mbox{span}(\bigcup_{i=1}^m e_i \cup \bigcup_{j \in J^c} e_{m+j})\backslash\{\boldsymbol{0}\}$, and all other balls.  Again, we may assume that the balls in the second case do not meet $\boldsymbol{0}$ and in the last case do not meet $\mbox{span}(\bigcup_{i=1}^m e_i \cup \bigcup_{j \in J^c} e_{m+j}) \cup \mbox{span}(\bigcup_{j \in J} e_{m+j})$.  Let $E := B(\boldsymbol{w}, r)$.  We prove each case in turn---for convenience of exposition, we prove the cases in the order first, third, second, and fourth.

\subsubsection{The first case:  balls centered at $\boldsymbol{0}$}  Let $R = e^{\lambda t}$.  Fix a large $R$.  Consider the rod $g_t^{-1} E$.  The directions $J$ are all expanded to a radius of $Rr$.  All other directions are expanding less or contracting.  As in Section~\ref{sec3typeMultpara}, there exists a minimal radius $\widetilde{\tau}(R)$ larger than which the analysis of the geometry is valid and we can choose $\widetilde{\tau}(R) =3 e^{\widetilde{\lambda}t}$; hence, we have $\lim_{R \rightarrow \infty} \widetilde{\tau}(R)/ Rr =0$, which implies we can ignore radius smaller than $\widetilde{\tau}(R)$.  As mentioned, caps $\Ca(\tau)$ are no longer small, but this does not affect the analysis of the geometry from Section~\ref{sec3typeMultpara} up to the inequality \begin{align*}\sum_{\boldsymbol{v} \in \La \backslash \{\boldsymbol{0}\} } A_R^E(\|\boldsymbol{v}\|) &\leq O(R^{-\ell}) + d \int_{\widetilde{\tau}}^{Rr}  (1+ \varepsilon)  \vol(B_{\Euc}(\boldsymbol{0}, 1))\frac{C(\widetilde{\tau})}{{\vol_{\s^d}(\s^d)}} ~\wrt \tau \\ & = O(R^{-\ell}) + d (1 + \varepsilon) \frac {\vol(B_{\Euc}(\boldsymbol{0}, 1))}{{\vol_{\s^d}(\s^d)}} C(\widetilde{\tau}) (Rr - \widetilde{\tau})\end{align*} where $C(\tau)=\vol_{\s^d}\Ca(\tau)$ and the $O(R^{-\ell})$ comes from $\tau < \widetilde{\tau}(R)$.  Now $C(\widetilde{\tau}) (Rr - \widetilde{\tau})$ is the volume of the rod with a (relatively) small hole missing.  Letting $R \rightarrow \infty$ and $\varepsilon \rightarrow 0$ yields our desired result:  \[\lim_{t \rightarrow \infty} \int_{K_{d}} \widehat{\chi}_E(g_t k \La) ~\wrt k \leq\vol_{\RR^d}(E).\]


\subsubsection{The third case:  balls centered at $\mbox{span}(\bigcup_{i=1}^m e_i \cup \bigcup_{j \in J^c} e_{m+j})\backslash\{\boldsymbol{0}\}$}  The proof is similar to the first case.  In any expanding directions $\bigcup_{j \in J} e_{m+j}$, the components of $\boldsymbol{w}$ are zero and hence $g_t^{-1} \boldsymbol{w}$ has at most expansion at a rate of $e^{\widetilde{\lambda}t}$.  Let $R = e^{\lambda t}$.  Fix a large $R$.  Let $\widetilde{\varepsilon}:= \widetilde{\varepsilon}(R):=\|g_t^{-1} \boldsymbol{w}\|/R$.  Hence, $\lim_{R \rightarrow \infty} \widetilde{\varepsilon}(R)= 0.$  Consequently, using the analogous proof as in the first case for a slightly larger rod (replacing $r$ with $(1 + 2\widetilde{\varepsilon})r$), we have \begin{align*}\sum_{\boldsymbol{v} \in \La \backslash \{\boldsymbol{0}\} } A_R^E(\|\boldsymbol{v}\|) &\leq O(R^{-\ell})+ d \int_{\widetilde{\tau}}^{\tau_+}  (1+ \varepsilon)  \vol(B_{\Euc}(\boldsymbol{0}, 1))\frac{C(\widetilde{\tau})}{{\vol_{\s^d}(\s^d)}} ~\wrt \tau \end{align*} where $\tau_+ := Rr (1+ 2\widetilde{\varepsilon})$ and $C(\widetilde{\tau}) (\tau^+ - \widetilde{\tau}) \rightarrow \vol_{\RR^d}(E)$ as $R \rightarrow \infty$.  This yields our desired result:  \[\lim_{t \rightarrow \infty} \int_{K_{d}} \widehat{\chi}_E(g_t k \La) ~\wrt k \leq\vol_{\RR^d}(E).\]



\subsubsection{The second case: balls centered in $\mbox{span}(\bigcup_{j \in J} e_{m+j}) \backslash\{\boldsymbol{0}\}$}  Of the indices in $J$ pick one, say $j_1$.  Let us first consider the special case that $\boldsymbol{w} = w e_{m+j_1}$ for some $w \neq 0$.  This index will play the role of $d$ from Section~\ref{sec2typeMultpara}.  Let $I=\{1, \cdots, m\} \cup \{m+j : j \in J^c\}$ and $R = e^{\lambda t}$.  For the second (and fourth) cases, we will assume an additional condition (which we later show does not affect the generality of our result):  for a fixed $\alpha \geq 1$, we only consider balls $E$ for which \begin{align}\label{eqnRadBnd}\frac{\dist\left(\boldsymbol{w}, \spn(\bigcup_{i \in I} e_i)\right) - r}{r}\geq \alpha\end{align} holds.  Recall that our ball $E$ is given by the sup norm---it is a $d$-cube. Its translate $E - \boldsymbol{w}$ has exactly one vertex $\boldsymbol{p}$ with all positive coordinates.  Let us change $E$ into a ``half-closed'' ball $F$ by the union of all the $d-1$ hyperfaces of the cube with $\boldsymbol{p} + \boldsymbol{w}$ as vertex.  Any half-closed ball will be constructed like this.  We will refer to $F$ and its $g_t$-translates as \textit{half-closed rods} or simply \textit{rods} if the context is clear.  Fix a large $R$.  Consider the rod $g_t^{-1} F$ and one of the $d-1$-dimensional faces that is normal to $e_{m+j_1}$---call this face $\F$ and note that it is a $d-1$-dimensional box.


Choose a large natural number $N$.  Partition the smallest side length of $\F$ into $N$ segments of length $L$.  For each of the other side lengths in $\F$, partition into segments whose length is nearest to $L$.  This partitions $\F$ into $\Nu(N)$ boxes with the same side lengths and, furthermore, each of which can be contained in a $d-1$-dimensional cube of side length $2L < 2/N$.  Let us index these little boxes by $k$.  The cartesian product of each of these little cubes with the $e_{m+j_1}$-th coordinate of $g_t^{-1} F$ are rods, which we make into half-closed rods in the way specified above.  This is a partition of $g_t^{-1} F$ such that there is only one direction, namely $e_{m+j_1}$, that is long.  To each element of this partition, cases two and four of Section~\ref{sec2typeMultpara} applies (the fact that each element is a half-closed rod as opposed to an open rod does not affect the proof in Section~\ref{sec2typeMultpara}).  Since this is a partition, elements are pairwise disjoint and we may sum over each element of the partition to obtain \begin{align*}\sum_{\boldsymbol{v} \in \La \backslash \{\boldsymbol{0}\} } A_R^F(\|\boldsymbol{v}\|) &\leq d \sum_{k=1}^{\Nu(N)} \int_{\tau_k^-}^{\tau_k^+}  (1+ \varepsilon)  \vol(B_{\Euc}(\boldsymbol{0}, 1))\frac{C_k(\tau_k^-)}{{\vol_{\s^d}(\s^d)}} ~\wrt \tau \\ & = d (1 + \varepsilon) \frac {\vol(B_{\Euc}(\boldsymbol{0}, 1))}{{\vol_{\s^d}(\s^d)}}  \sum_{k=1}^{\Nu(N)} C_k(\tau_k^-) (\tau_k^+ - \tau_k^-) \end{align*} where $\tau_k^-$ and $\tau_k^+$ are, up to $O(R^{-1})$, the minimum and maximium radii such that $\tau \s^d$ meets the $k$-th partition element and $C_k(\tau)$ is the volume of the cap of the $k$-th element, i.e. $C_k(\tau) = \vol_{\tau \s^d}(\Ca_k(\tau))$ where $\Ca_k(\tau)$ is the intersection of $k$-th partition element with $\tau \s^d$.  Within $O(R^{-1})$, $\tau_k^+ - \tau_k^-$ is the length of the rod $g_t^{-1} F$ along the $e_{m+j_1}$ direction.   Now $C_k(\tau_-) (\tau_+ - \tau_-)$ is the volume of an element that has length along $e_{m+j_1}$ within $O(R^{-1})$ of the length along $e_{m+j_1}$ of $g_t^{-1} F$, but with cross-section volume $C_k(\tau_-)$.  Since in the second case (\ref{eqnRadBnd}) holds, a direct calculation (using trigonometry) gives that \begin{align}\label{eqnLargeCapEst}C_k(\tau_-) \leq \widetilde{\gamma} \frac{\vol_{\RR^{d-1}}(\B_k)}{\sin^{d-1}(\pi/2 - \arcsin(1/\alpha))} \end{align} where $\widetilde{\gamma} >1$ is a number depending only $N$ and $\alpha$ such that $\widetilde{\gamma} \searrow 1$ as $N, \alpha \rightarrow \infty$ and $\B_k$ is the interesection of the $d-1$-dimensional hyperplane normal to $e_{m+j_1}$ with the $k$-th partition element.  Consequently, for large $N, \alpha$, \[\sum_{k=1}^{\Nu(N)} C_k(\tau_k^-) (\tau_k^+ - \tau_k^-)\] is arbitrarily close to $\vol_{\RR^d}(E)$.  As $ \frac {\vol(B_{\Euc}(\boldsymbol{0}, 1))}{{\vol_{\s^d}(\s^d)}}= \frac 1 {d+1}$, letting $R \rightarrow \infty$ and $\varepsilon \rightarrow 0$, we have our desired result for the special case:  \[\lim_{t \rightarrow \infty} \int_{K_{d}} \widehat{\chi}_E(g_t k \La) ~\wrt k \leq\vol_{\RR^d}(E),\] up to the restrictions that the balls are now half-closed and that (\ref{eqnRadBnd}) must hold.  Likewise, we have the same conclusion for $\boldsymbol{w} = w e_{m+j}$ for any $j \in J$.

We now consider the second case in general.  We may assume that $\boldsymbol{w} \in \mbox{span}(\bigcup_{j \in J} e_{m+j})=:\Pl_J$ but not in any of the coordinate axes.  Let $\boldsymbol{q}:=\boldsymbol{q}(R)$ denote the point of $\overline{g_t^{-1} F}$ with smallest Euclidean norm.  By convexity, it is easy to see that the point $\boldsymbol{q}$ is unique (for fixed $R$) and that $\boldsymbol{q} \in \mbox{span}(\bigcup_{j \in J} e_{m+j})$.  We remark that $\boldsymbol{q}$ is an eigenvector of $g_t^{-1}$ and thus the direction of $\boldsymbol{q}$ is fixed for all $R$.  Let $\|\cdot\|_J$ be the sup norm in $\Pl_J$.  Rotate a coordinate axis to the direction of $\boldsymbol{q}$---doing this to a half-closed $\|\cdot\|_J$-ball of radius $\beta$ yields a rotated half-closed $\ell$-cube $\Cu(\beta)$ with side length $2 \beta$.  Cover $F \cap \Pl_J$ by a partition of affine translates of $\bigcup_t \Cu_t(\beta)$ where $\beta>0$ is a constant so small that $\vol_{\RR^\ell}(F \cap \Pl_J)$ is less than but as close to $\vol_{\RR^\ell}(\bigcup_t \Cu_t(\beta))$ as desired.  For each $\Cu_t(\beta)$ take the cartesian product with the other directions of $F$ to obtain $\widetilde{F}_t(\beta)$.  Then $\vol_{\RR^d}(F)$ is less than but as close to $\vol_{\RR^d}(\bigcup_t \widetilde{F}_t(\beta))$ as desired.  Choose $t$ so large that $e^{\lambda t}\beta$ is larger than the $R$ chosen in the special case (where the center is on the axis) above---this gives us a much larger $R$ for this, the second case in general.  Now the special case holds for each $\widetilde{F}_t(\beta)$ and applying it to each and summing over the partition and noting that the volume of the partition is arbitrarily close to $\vol_{\RR^d}(E)$ yields the second case in general:  \[\lim_{t \rightarrow \infty} \int_{K_{d}} \widehat{\chi}_E(g_t k \La) ~\wrt k \leq\vol_{\RR^d}(E),\] up to the restrictions that the balls are now half-closed and that (\ref{eqnRadBnd}) must hold.


\subsubsection{The fourth case: all other balls}  This is an adaption of the second case.  The difference is that $\boldsymbol{q} \notin \Pl_J$.  Let $\|\cdot\|_I$ be the sup norm in the $\spn(\bigcup_{i \in I} e_{i}) =: \Pl_I$ and let $\boldsymbol{q}_I$ and $\boldsymbol{q}_J$ be the orthogonal projections of $\boldsymbol{q}$ onto $\Pl_I$ and $\Pl_J$, respectively.  Then \[\frac{\|\boldsymbol{q}_I(R)\|_I}{\|\boldsymbol{q}_J(R)\|_J} = 0\] as $R \rightarrow \infty$.  Consequently, for $R$ large enough, (\ref{eqnLargeCapEst}) holds and thus the proof of the second case also applies to this case, allowing us to conclude:   \[\lim_{t \rightarrow \infty} \int_{K_{d}} \widehat{\chi}_E(g_t k \La) ~\wrt k \leq\vol_{\RR^d}(E),\] up to the restrictions that the balls are now half-closed and that (\ref{eqnRadBnd}) must hold. 

\subsubsection{Finishing the second and fourth cases}  We wish to prove the second and fourth cases for the balls defined for the first and third cases (i.e. in terms of the product of Euclidean norms).  To remove the restriction of half-closed rods, consider the measure zero boundaries of the half-closed rods at each stage.  Using~\cite[Lemma~3.5]{AGT1} to approximate this measure zero set and the method of handling the null term from~\cite[Section~3.4]{AGT1}, we can remove this restriction.  To remove the restriction given by (\ref{eqnRadBnd}), we note that~\cite[Lemmas~3.1 and~3.5]{AGT1} apply to the balls of the second and fourth case with the restriction (\ref{eqnRadBnd}) because $\alpha$ is fixed and the ball given by the product of the Euclidean norms not do meet $\Pl_I$.  This is all that is needed to apply~\cite[Section~3.4]{AGT1}.  Doing so allows us to conclude \[\lim_{t \rightarrow \infty} \int_{K_{d}} \widehat{\chi}_E(g_t k \La) ~\wrt k \leq\vol_{\RR^d}(E),\] where $E$ is a ball in the same norm as for the first and third case.

\subsection{Finishing the proof of Theorem~\ref{theorem:siegel:equidistUpper}}  For each mutiparameter type, we apply~\cite[Section~3.4]{AGT1} without change. 


\section{Appendix}  We prove Lemma~\ref{lemmInfiniteVolCusp}.  Recall that $\|\boldsymbol{v}\|_{\pr} := \prod_{i=1}^\ell |v_i|$; let $\ell=m$ in this section.  Let \[S := \{ \boldsymbol{v} \in \RR^m :\|\boldsymbol{v}\|_{\pr} \leq 1\}.\]

\begin{lemm}\label{lemmInfiniteVolCusp}  Let $\boldsymbol{w}$ be on a great sphere for $\s^{m-1}$.  Let $A:=B(\boldsymbol{w}, r) \cap \s^{m-1}$ for some $r>0$.  Then 
$$\vol_{\RR^m}(\Co A \cap S) = \infty.$$  


 
\end{lemm}
\begin{proof}  For $m=1$, a great sphere is simply the intersection of an axis with the circle.  Elementary calculus gives the result.  


We may assume that $m \geq 2$.  Without loss of generality, we may assume that $r$ is small.  There are two cases.  First assume that $\overline{A}$ does not meet any coordinate axes.  Then there exists exactly one coordinate in which the points in $\overline{A}$ may have small absolute value.  By reordering indices if necessary, we may assume that the $m$-th coordinate is the one that has small absolute values.  In the other directions, the absolute values are bounded away from $0$.  In other words, given a constant $c>0$, we have 
$$\prod_{i=1}^{m-1}|v_i| \geq c$$ 
\noindent for all $\boldsymbol{v} \in \overline{A}$.  Note that $\vol_{\s^{m-1}}(A) =O(r^{m-1})$ and that 
$$\prod_{i=1}^{m-1}|v_i| = O(\vol_{\s^{m-1}}(A))$$
\noindent for all $\boldsymbol{v} \in \overline{A}$ (because $r$ is small).  Consequently, for large $\tau$, we have that 
$$\prod_{i=1}^{m-1}|\tau v_i| = O(\vol_{\tau \s^{m-1}}(\tau A)).$$ 
\noindent Perhaps by cutting off the part of the cone nearest to the origin, we have that  $\Co A \cap S$ is the graph of the function over $A$ determined by $\prod_{i=1}^{m}|x_i| =1$, giving us that $ \vol_{\tau \s^{m-1}}(\tau A \cap S) |\tau v_m| =O(1)$.  This implies that $|v_m| \tau^m \leq O(1)$.  Riemann integration now gives
$$\vol_{\RR^m}(\Co A \cap S) = const \int_{1}^\infty \vol_{\tau \s^{m-1}}(\tau A \cap S)~\wrt \tau =  const \int_{1}^\infty \frac 1 {|\tau v_m|}~\wrt \tau \geq const \int_{1}^\infty \tau^{m-1}~\wrt \tau = \infty.$$
\noindent  We note that $const$ depends on how close $A$ is to a coordinate axis.


The other case is when $\overline{A}$ meets coordinate axes.  Since $r$ is small, it may only meet one.  Pick an open ball $\widetilde{B} \subset A$ that avoids the axis and apply the previous proof.
 
\end{proof}

\section{Concluding remarks and generalisations}

\noindent As we have noted in \cite{AGT1}, Theorem \ref{theorem:siegel:equidist} is reminiscent of some results in the literature, for instance the work of Eskin-Margulis-Mozes, where the authors average over a different compact group. Closest to our work is the result of Kleinbock-Margulis who provide a proof in \cite{KM} Corollary A.8 of a very general averaging result for one parameter flows, or more precisely, flows with the property ``EM" on quotients of semisimple groups by irreducible lattices. We refer the reader to \cite{KM} for precise statements and the definition of the property ``EM". The main tool in loc. cit. is exponential decay of correlation coefficients of H\"{o}lder vectors. In \cite{AGT1}, we provide a new, elementary proof of the above theorem for one parameter flows on $\SL_{n}(\R)/\SL_{n}(Z)$. On the other hand, as we have mentioned, the higher rank averaging result in the present paper is new.

Several authors have investigates variations of Dirichlet's theorem in the context of number fields. Let $k$ be a number field which we assume to be totally real for convenience, let $S$ be the set of infinite places of $k$ and let $O_S$ be the ring of $S$-integers of $k$. For $x \in k^n$, we define the $S$ height to be
$$ h_{S}(x) := \prod_{v \in S} \max_{1 \leq i \leq n}\{|x_i|_v\} $$
\noindent where $|~|_v$ denotes the $v$-adic valuation on the completion $k_v$ of $k$. In \cite{Burger}, Burger proves\footnote{He actually proves his result for matrices with entries in $k_v$} that for every $x_v \in k_v$ and $Q > c(k)^{(1+n)/n}$, there exist $q \in O^{n}_{S}$ and $p \in O_S$ such that $h_{S}(q) \leq Q$ and
$$ \prod_{v \in S}|x_v q - p|_{v} \leq c(k)^{1+n}Q^{-n}$$
\noindent where $c(k) = \Delta^{1/2d}_{k}$, $d$ is the degree of the number field and $\Delta_k$ is the discriminant of $k$. Using the methods of \cite{AGT1} and the above mentioned Theorem of Kleinbock-Margulis applied to 
$$ G = \prod_{v \in S}\SL_{n+1}(k_v), \Gamma = \SL_{n+1}(O_S)$$
\noindent and $g_t = (\diag(e^{nt}, e^{-t}, \dots, e^{-t}))_{v \in S}$ one can obtain equidistribution of approximates for the analogue of Dirichlet's theorem for number fields as above. Other generalisations of Dirichlet's theorem have been established by W. M. Schmid\cite{Schmidt-num}, Q\^{u}eme \cite{Queme}, Hattori \cite{Hattori}, see also \cite{EGL} where another analogue of Dirichlet's theorem and an analogue of badly approximable vectors in number fields is investigated. It would not be difficult to obtain equidistribution of approximates in each of these cases with appropriate choices of $G$, $\Gamma$ and $g_t$. However, multplicative, weighted and versions for more places ($p$-adic) do not follow from the above methods. It would be an interesting problem to generalise our higher rank result in this paper to include other semisimple Lie groups and also to include finite, i.e. $p$-adic places.

\end{document}